\documentclass[11pt,oneside,reqno]{amsart}

\usepackage{graphicx}
\usepackage{spectralsequences}
\usepackage[utf8]{inputenc}
\usepackage{xfrac}   
\usepackage{mathtools}
\usepackage{fancyhdr}
\usepackage{amssymb}
\usepackage{amsthm}
\usepackage{enumitem}
\usepackage{wrapfig}
\usepackage{hyperref}
\usepackage{mathrsfs}
\usepackage[british]{babel}
\usepackage[text={16cm,22cm}]{geometry}
\usepackage{xcolor}
\usepackage{amsmath}
\usepackage{amsfonts}
\usepackage{spectralsequences}
\usepackage{url}
\usepackage{ stmaryrd }
\usepackage{graphics}
\usepackage{amssymb}
\usepackage{listings}
\usepackage[disable]{todonotes}
\usepackage{ifthen}
\usepackage{quiver}
\usepackage[LGR,T1]{fontenc}
\usepackage[autostyle=true]{csquotes}
\usepackage{xparse,xfp}

\ExplSyntaxOn
\NewDocumentCommand{\defineconstant}{mm}
 {
  \cs_new:Npx #1 { \fp_eval:n { #2 } }
 }
\ExplSyntaxOff

\newcommand{\Z}{\mathbb{Z}}
\newcommand{\R}{\mathbb{R}}
\newcommand{\Q}{\mathbb{Q}}

\newcommand{\N}{\mathbb{N}}
\newcommand{\F}{\mathbb{F}}
\newcommand{\Sph}{\mathbb{S}}

\newcommand{\A}{\mathcal{A}}
\newcommand{\lto}{\longrightarrow}

\DeclareMathOperator{\im}{im}
\DeclareMathOperator{\coker}{coker}

\DeclareMathOperator{\id}{id}

\DeclareMathOperator{\D}{Diff}

\DeclareMathOperator{\Sp}{Sp}
\DeclareMathOperator{\Spc}{Spc}

\DeclareMathOperator{\Ext}{Ext}
\DeclareMathOperator{\Sq}{Sq}
\DeclareMathOperator{\Diff}{Diff}
\DeclareMathOperator{\sk}{sk}

\usepackage{tikz-cd}
\tikzcdset{
arrow style=tikz,
diagrams={>={Stealth[scale=0.9]}}
}
\usetikzlibrary{arrows}
\tikzset{commutative diagrams/.cd,arrow style=tikz,diagrams={>=latex'}}
\usepackage{tikz}
\usetikzlibrary{arrows,chains,matrix,positioning,scopes}

\makeatletter
\tikzset{join/.code=\tikzset{after node path={%
\ifx\tikzchainprevious\pgfutil@empty\else(\tikzchainprevious)%
edge[every join]#1(\tikzchaincurrent)\fi}}}
\makeatother

\tikzset{>=stealth',every on chain/.append style={join},
         every join/.style={->}}
         
\makeatletter
\newcommand{\colim@}[2]{%
  \vtop{\m@th\ialign{##\cr
    \hfil$#1\operator@font colim$\hfil\cr
    \noalign{\nointerlineskip\kern1.5\ex@}#2\cr
    \noalign{\nointerlineskip\kern-\ex@}\cr}}%
}
\newcommand{\colim}{%
  \mathop{\mathpalette\colim@{\rightarrowfill@\textstyle}}\nmlimits@
}
\makeatother
\makeatletter
\newcommand{\hocolim@}[2]{%
  \vtop{\m@th\ialign{##\cr
    \hfil$#1\operator@font hocolim$\hfil\cr
    \noalign{\nointerlineskip\kern1.5\ex@}#2\cr
    \noalign{\nointerlineskip\kern-\ex@}\cr}}%
}
\newcommand{\hocolim}{%
  \mathop{\mathpalette\hocolim@{\rightarrowfill@\textstyle}}\nmlimits@
}
\makeatother

\newtheorem{theorem}{Theorem}[section]
\newtheorem{lemma}[theorem]{Lemma}
\newtheorem{corollary}[theorem]{Corollary}
\newtheorem{proposition}[theorem]{Proposition}

\theoremstyle{definition}
\newtheorem{definition}[theorem]{Definition}
\newtheorem{constr}[theorem]{Construction}
\newtheorem{conv}[theorem]{Convention}

\theoremstyle{remark}
\newtheorem{remark}[theorem]{Remark}

\theoremstyle{theorem}

\newlist{legal}{enumerate}{10}
\setlist[legal]{label*=\arabic*.}

\numberwithin{equation}{section}

\makeatletter
\newcommand{\manuallabel}[2]{\def\@currentlabel{#2}\label{#1}}
\makeatother

\makeatletter
\newcommand*{\defeq}{\mathrel{\rlap{%
                     \raisebox{0.3ex}{$\m@th\cdot$}}%
                     \raisebox{-0.3ex}{$\m@th\cdot$}}%
                     =}
\makeatother
    \newcommand\quotient[2]{
        \mathchoice
            {
                \text{\raise1ex\hbox{$#1$}\Big/\lower1ex\hbox{$#2$}}%
            }
            {
                #1\,/\,#2
            }
            {
                #1\,/\,#2
            }
            {
                #1\,/\,#2
            }
    }

\newcommand\restr[2]{{
  \left.\kern-\nulldelimiterspace 
  #1 
  \vphantom{\big|} 
  \right|_{#2} 
  }}

\DeclarePairedDelimiter\floor{\lfloor}{\rfloor}
\tikzset{
  symbol/.style={
    draw=none,
    every to/.append style={
      edge node={node [sloped, allow upside down, auto=false]{$#1$}}}
  }
}
\newsavebox{\pullback}
\sbox\pullback{%
\begin{tikzpicture}%
\draw (0,0) -- (1ex,0ex);%
\draw (1ex,0ex) -- (1ex,1ex);%
\end{tikzpicture}}
\newsavebox{\pushout}
\sbox\pushout{%
\begin{tikzpicture}%
\draw (0,0) -- (0ex,1ex);%
\draw (0ex,1ex) -- (1ex,1ex);%
\end{tikzpicture}}
\tikzset{
  symbol/.style={
    draw=none,
    every to/.append style={
      edge node={node [sloped, allow upside down, auto=false]{$#1$}}}
  }
}

\begin{document}
\title{Splitting the Madsen-Tillmann Spectra $MT\theta_n$}
\author{Jonathan Sejr Pedersen}
\address{Department of Mathematics, University of Toronto, Toronto, Ontario, Canada}
\email{jsejrp@math.toronto.edu}
\author{Andrew Senger}
\address{Department of Mathematics, Harvard University, Cambridge, MA, USA}
\email{senger@math.harvard.edu}
\date{\today}

\begin{abstract}
    We prove that the Madsen-Tillmann spectrum $MT\theta_n$ splits into the sum of spectra $\Sigma^{-2n}MO\langle n+1 \rangle \oplus \Sigma^{\infty-2n}\R P^\infty_{2n}$ after Postnikov trunctation $\tau_{\leq \ell}$ for $\ell = \lfloor \frac{n}{2} \rfloor - 6$. To accomplish this, we prove that the connecting map in a certain fiber sequence is nullhomotopic in this range by an Adams filtration argument. As an application, we compute $H_2(B\Diff(W^{2n}_{g},D^{2n});\Z)$ up to extensions for $n \geq 16$ and $g \geq 7$.
%
\end{abstract}
\maketitle
\setcounter{tocdepth}{1}
\tableofcontents

\section{Introduction}
Let $W^{2n}_g = \#_g (S^n\times S^n)$ denote the $g$-fold connected sum of $S^n\times S^n$ and choose a fixed closed disc $D^{2n}\subset W^{2n}_g$. Let 
$\D^+(W^{2n}_g)$ be the topological group (in the $C^\infty$-topology) of orientation preserving diffeomorphisms, and $\D(W^{2n}_g,D^{2n})$ the subgroup of those diffeomorphisms that fix pointwise an open neighborhood of the disc.

Parameterized Pontryagin-Thom theory \cite[Theorem 4.1]{galatius2019moduli} provides, for $2n\neq 4$, a map
\begin{align*}
    \alpha\colon B\D(W^{2n}_g,D^{2n}) \to \Omega^\infty_0 MT\theta_n
 \end{align*}
which is an isomorphism on $H_*$ for $* \leq \frac{g-3}{2}$. Here $MT\theta_n$ is the Madsen-Tillmann spectrum constructed as the Thom spectrum of the virtual bundle $-\theta_n^*\gamma_{2n}\to \tau_{>n}BO(2n)$ where $\theta_n\colon \tau_{>n} BO(2n)\to BO(2n)$ denotes the $n$-connected cover.

Although the homotopy types of the source and target of $\alpha\colon B\D(W^{2n}_g,D^{2n}) \to \Omega^\infty_0 MT\theta_n$ differ, this is nevertheless a powerful tool in the study of $B\D(W^{2n}_g,D^{2n})$.
For example, as the first integral homology group computes the abelianization of the fundamental group, this relates the abelianization of the mapping class group to the first homology group of $\Omega_0 ^\infty MT\theta_n$, which in turn is isomorphic to $\pi_1 MT\theta_n$.

This was exploited by Galatius and Randal-Williams to compute the abelianizations of these mapping class groups in \cite{galatius2015abelian}.
A key feature of their technique was the stabilization map $s\colon \Sigma^{2n}MT\theta_n \to MO\langle n +1\rangle$ where $MO\langle n+1 \rangle$ is the Thom spectrum of the bundle $\theta_n^*\gamma \to \tau_{>n}BO = \tau_{\geq n+1}BO$.\footnote{Note that our convention for the indexing of $MO \langle n \rangle$ differs from that of Galatius and Randal-Williams. See the conventions section for more details.}
Outside of low-dimensional cases, they show that $\pi_1 MT\theta_n$ splits as $\pi_{2n+1} MO\langle n+1 \rangle \oplus H_1 (\mathrm{Aut}(Q_{W^{2n} _g}))$, where $Q_{W^{2n} _g}$ is a certain quadratic form on the middle homology group of $W^{2n}_g$. Moreover, outside of finitely many cases the groups $\pi_{2n+1} MO \langle n+1 \rangle$ are isomorphic to $\coker(J)_{2n+1}$ by \cite{stolz2007hochzusammenhangende,burklund2019boundaries,burklund2020highdimensional}.

This splitting was also used in work of Krannich--Reinhold that constrains the signatures of bundles of highly-connected manifolds over surfaces \cite{manueljens}. Information about the higher homotopy groups of $MT\theta_n$ could allow an extension of their methods to manifold bundles over a more general base.

\subsection{Results}
Our main theorem strengthens the splitting of $\pi_1 MT\theta_n$ by Galatius and Randal--Williams to a splitting of the spectrum $MT\theta_n$ in a range.
%
%
%

\begin{theorem}\label{thm:mainthm}
    The fiber sequence
    \begin{align*}
F_n \lto \Sigma^{2n}MT\theta_n \xlongrightarrow{s} MO\langle n+1 \rangle,
\end{align*}
becomes split after Postnikov truncation $\tau_{\leq \ell }$ where $\ell = 2n+\floor{\frac{n}{2}}-6$, and thus there is an equivalence of spectra 
    \begin{align*}
       \tau_{\leq \floor{\frac{n}{2}}-6} (MT\theta_n) &\simeq \tau_{\leq \floor{\frac{n}{2}}-6} \Sigma^{-2n} MO \langle n+1 \rangle \oplus \tau_{\leq \floor{\frac{n}{2}}-6} \Sigma^{-2n}F_n \\
       &\simeq \tau_{\leq \floor{\frac{n}{2}}-6} \Sigma^{-2n} MO \langle n+1 \rangle \oplus \tau_{\leq \floor{\frac{n}{2}}-6} \Sigma^{\infty-2n}\R P^{\infty}_{2n}.
    \end{align*}
\end{theorem}

\begin{remark}
We will see in Lemma \ref{lem:MTcofibseq} below, why $F_n$ is equivalent to $\Sigma^\infty \R P_{2n} ^\infty$ in the range of the theorem.
\end{remark}

The proof of Theorem \ref{thm:mainthm} will be the subject of Sections \ref{sec:cofiber} and \ref{sec:splitting2}.
As a corollary, we obtain a calculation of the second integral homology of $\Omega_0^\infty MT\theta_n$ in Section \ref{sec:H2}. Combining this with parameterized Pontryagin-Thom theory yields the following:

\begin{corollary}
For $n\geq 16$ and $g\geq 7$ we have a non-canonical isomorphism
\begin{align*}
   H_2(B\D(W^{2n}_g,D^{2n});\Z)  &\cong H_2(\pi_{2n+1} MO \langle n+1 \rangle\oplus \pi_{2n+1}(\Sigma^{\infty} \R P^\infty _{2n})) \\
   &\oplus \pi_{2n+2} MO \langle n+1 \rangle\oplus \pi_{2n+2}(\Sigma^\infty \R P^\infty _{2n}),
\end{align*}
where $H_2(G)$ for a group $G$ denotes the second group homology with coefficients in the trivial $G$-module $\mathbb{Z}$.
\end{corollary}
The various terms appearing on the right hand side are quite accessible in many cases by work of \cite{burklund2019boundaries,burklund2020highdimensional, sengeryiyu2022} and \cite{hoo1965some}.




\subsection{Sketch of the proof}
The main idea of the proof is as follows: first, there is a map $\Sigma^\infty \R P^\infty _{2n} \to F_n$ which is an equivalence upon taking $\tau_{\leq 3n-1}$, so that there is a cofiber sequence of $3n$-truncated spectra
\[\tau_{\leq 3n} MO \langle n + 1 \rangle \to \tau_{\leq 3n} \Sigma^{\infty+1} \R P^\infty _{2n} \to \tau_{\leq 3n} \Sigma^{2n+1} MT\theta_n.\]

To split a truncation of $MT\theta_n$, it will therefore suffice to provide a nullhomotopy of a truncation of the attaching map
\[\tau_{\leq 3n} MO \langle n + 1 \rangle \to \tau_{\leq 3n} \Sigma^{\infty+1} \R P^\infty _{2n}.\]
To do this, we will first use the stable cohomotopy Euler class to factor it through the truncation of an auxiliary spectrum $Y$, which has two advantages over $\Sigma^{\infty+1} \R P^\infty_{2n}$.
Firstly, $Y$ is $2$-complete and secondly, the $2$-primary Adams spectral sequence of $Y$ admits a vanishing line of slope $\frac{1}{2}$.

To prove that a truncation of this attaching map is nullhomotopic, it will therefore suffice to prove that it is of high $2$-primary Adams filtration. To do this, we factor $\tau_{2n+k} MO \langle n + 1 \rangle \to \tau_{\leq 2n+k} \Sigma^{\infty+1} \R P^\infty _{2n}$ through the map $\tau_{2n+k} MO \langle n + 1 \rangle \to \tau_{\leq 2n+k} MO \langle n-k+1\rangle$, which is known by the Thom isomorphism and the computations of Stong \cite{stongBOk} to be of high Adams filtration in a range.

\subsection{Conventions}
Recall the functors (defined for spaces or spectra) $\tau_{\geq n}, \tau_{\leq n}$, which are equipped with natural transformations $\tau_{\geq n} X \to X \to \tau_{\leq n}X$ and satisfy:
\[
\pi_*(\tau_{\geq n}X) = \begin{cases}
    \pi_*(X) & \text{if } *\geq n \\
    0 & \text{if } *<n,
\end{cases} \quad \pi_*(\tau_{\leq n}X) = \begin{cases}
    \pi_*(X) & \text{if } *\leq n \\
    0 & \text{if } *>n.
\end{cases}
\]

We say that a space or spectrum is $n$-\emph{connective} if $X \simeq \tau_{\geq n} X$ and $n$-\emph{connected} if $X \simeq \tau_{> n} X$.
A map of spaces or spectra is $n$-connected (or an $n$-equivalence) if the fiber $F$ is $(n-1)$-connected, i.e. $F \simeq \tau_{>n-1} F$.

Notably, we write $MO\langle n \rangle$ for the Thom spectrum of $\tau_{\geq n} BO \to BO$, following the convention in homotopy theory that $BO\langle n \rangle = \tau_{\geq n} BO$, the $n$th \emph{connective} cover of $BO$. This convention agrees with that used in the papers \cite{burklund2019boundaries,burklund2020highdimensional,sengeryiyu2022}, and contrasts with that of Galatius and Randal-Williams in \cite{galatius2015abelian}, who use $MO\langle n \rangle$ for the Thom spectrum of $\tau_{> n} BO \to BO$.

\subsection{Acknowledgments}
The authors would like to thank Alexander Kupers, the advisor of the first author, for invaluable help and discussions, as well for bringing the authors together to make this collaboration possible.

During the course of this work, Andrew Senger was supported by NSF grant DMS-2103236.

\section{The cofiber sequence for $MT\theta_n$} \label{sec:cofiber}

In this section, we first relate $MT\theta_n$ to $MO\langle n + 1 \rangle$ and $\Sigma^\infty \R P^{\infty} _{2n}$ in a range via a cofiber sequence. We then factor the attaching map of this cofiber sequence through an auxiliary spectrum $Y$ which we will find simpler to work with than $\Sigma^\infty \R P^\infty_{2n}$.
These arguments are inspired by those of \cite[Section 5.1]{galatius2015abelian}.
Finally, we prove that $Y \simeq Y^{\wedge} _2$, reducing this attaching map to one that can be studied via the $2$-primary Adams spectral sequence.


\begin{remark}
    The spectrum $MT\theta_n$ comes from manifold theory, where the convention is that $n$ refers to the connect\emph{ed} cover, while the spectrum $MO\langle n \rangle$ comes from homotopy theory where $n$ refers to the connect\emph{ive} cover. This distinction accounts for the shift in index.
\end{remark}

\begin{definition}
    Letting $\theta_{2n,k} : \tau_{> k} BO(2n) \to BO(2n)$ denote the $k$-connected Whitehead cover and $\gamma_{2n} \to BO(2n)$ denote the tautological vector bundle, we let
    \[MT\theta_{2n,k} \defeq \operatorname{Th}(-\theta_k^*\gamma_{2n} \to \tau_{>k}BO(2n)).\]
    We then have $MT\theta_n = MT\theta_{2n,n}$.
\end{definition}

\begin{lemma}\label{lem:MTcofibseq}
    For each $0\leq k \leq n$ there are cofiber sequences in the category of $(2n+k)$-truncated spectra $\mathrm{Sp}_{\leq 2n+k}$ (resp. in $\mathrm{Sp}_{\leq 2n+k-1}$)
    and natural maps
    \begin{equation*}
        \begin{tikzcd}
          \tau_{\leq 2n+k}MO\langle k+1 \rangle \rar{}\dar{} & \tau_{\leq 2n+k}\Sigma^{\infty+1}\R P^\infty_{2n} \arrow[d] \rar{} & \tau_{\leq 2n+k}\Sigma^{2n+1}MT\theta_{2n,k} \dar{} \\
          \tau_{\leq 2n+k-1}MO\langle k \rangle \rar{} & \tau_{\leq 2n+k-1}\Sigma^{\infty+1}\R P^\infty_{2n} \rar{} &  \tau_{\leq 2n+k-1} \Sigma^{2n+1}MT\theta_{2n,k-1}
        \end{tikzcd}
    \end{equation*}
    where the middle vertical map is the truncation.
\end{lemma}
\begin{proof}
We first construct a map $SO/SO(2n) \to F_n$ and then prove that it is highly connected. For this, we inspect the diagram
\begin{equation}\label{eq:BOsquare}
    \begin{tikzcd}
        SO/SO(2n) \rar{} \dar{} \arrow[dr, phantom, "\usebox\pullback" , very near start, color=black] & \tau_{>k}BO(2n)\rar{} \dar{}  & BO(2n) \dar{} \\
        *\rar{}  & \tau_{>k}BO \rar{} & BO \arrow[loop right, "-1"]
    \end{tikzcd}
\end{equation}
where the left hand square is a homotopy pullback square by the fact that $BO(2n) \to BO$ induces an isomorphism on $\pi_k$ for $k \leq n$. It follows from Blakers-Massey that the left hand square is also a $(2n+k+1)$-homotopy pushout, i.e. that the map from the homotopy cofiber of $SO/SO(2n) \to \tau_{> k} BO(2n)$ to $\tau_{> k} BO$ is an $(2n+k+1)$-equivalence.


Postcomposing this diagram with the map $BO \xrightarrow{-1} BO$ and taking Thom spectra, we obtain a diagram of the form:
\begin{equation}
    \begin{tikzcd}
        \Sigma^\infty _+ SO/SO(2n) \rar{} \dar{} & \Sigma^{2n} MT\theta_{2n,k} \rar{} \dar{}  & \Sigma^{2n} MT\theta_{2n,0}\dar{} \\
        \mathbb{S}\rar{}  & MO\langle k + 1 \rangle \rar{} & MO
    \end{tikzcd}
\end{equation}
Here, we have used the following equivalences of spaces over $BO$ to identify the Thom spectra appearing on the bottom line with $MO \langle k + 1 \rangle$ and $MO$:
\begin{center}
    \begin{tikzcd}
        \tau_{> k} BO \ar[rr, "-1"] \ar[rd,"-1"] & & \tau_{> k} BO & & BO \ar[rr, "-1"] \ar[rd, "-1"] & & BO \\
        & BO \ar[ur, "1"] & & & & BO. \ar[ur, "1"] & & 
    \end{tikzcd}
\end{center}
Since the Thom spectrum functor preserves colimits and connectivity, it follows that the map from the cofiber of $\Sigma^\infty SO/ SO(2n) \to \Sigma^{2n} MT\theta_{2n,k}$ to $MO\langle k + 1 \rangle$ is $(2n+k+1)$-connected.
Applying $\tau_{\leq 2n+k}$, we obtain a cofiber sequence of $(2n+k)$-truncated spectra
\[\tau_{\leq 2n+k} \Sigma^\infty SO/ SO(2n) \to \tau_{\leq 2n+k} \Sigma^{2n} MT\theta_{2n,k} \to \tau_{\leq 2n+k} MO \langle k+1 \rangle.\]

To complete the proof of the lemma, it suffices to rotate twice and apply \cite[Theorem 3]{whitehead45sphere} to see that \[\tau_{\leq 2n+k} \Sigma^\infty SO/ SO(2n) \simeq \tau_{\leq 2n+k} \Sigma^\infty \mathbb{R}P^\infty _{2n}\] and note that everything above was natural as $k$ varies.
\end{proof}

\begin{corollary} \label{cor:getSplit}
        For $\ell \leq 2n+k$, a nullhomotopy of the map $\tau_{\leq \ell} MO\langle k + 1 \rangle \to \tau_{\leq \ell} \Sigma^{\infty + 1} \R P^\infty_{2n}$ gives rise to a splitting
    \[\tau_{\leq \ell -1} \Sigma^{2n} MT\theta_{2n,k} \simeq \tau_{\leq \ell-1}MO \langle k + 1 \rangle \oplus \tau_{\leq \ell-1} \Sigma^\infty \R P^\infty _{2n}\]
\end{corollary}

\begin{proof}
    The functor $\tau_{\leq \ell}$ takes cofiber sequences of $(2n+k)$-truncated spectra to cofiber sequences of $\ell$-truncated spectra, so this follows from the fact that the cofiber of $A \xrightarrow{0} B$ is $B \oplus \Sigma A$.
\end{proof}

Next, we define the spectrum $Y$ through which we are going to factor the attaching map $\tau_{\leq 2n+k} MO \langle k + 1 \rangle \to \tau_{\leq 2n+k} \R P^\infty _{2n}$. Our construction uses the \emph{stable cohomotopy Euler class} \cite[Section 2]{stolz1989level}.

 \begin{constr}
 Given a rank $n$-vector bundles $E \to B$ and a virtual vector bundle $F \to B$, there is a map of Thom spectra $\operatorname{Th}(F) \to \operatorname{Th}(E\oplus F)$ induced by the zero section of $E$. Letting $F = -E$, this becomes a map $\operatorname{Th}(-E) \to \Sigma^{\infty}_+B$. Composing with the map induced by $B \to \ast$ gives a map $\operatorname{Th}(-E) \to \Sph$ which is the \emph{stable cohomotopy class} $e^s\in [\operatorname{Th}(-E),\Sph^0]$ of $E$. The only property of this construction that we will use is that under the Hurewicz map  $\Sph \to H\Z$ the stable cohomotopy Euler class maps to the (twisted) Euler class of $E$:
\begin{align*}
    e^s \mapsto e(E) \in [\operatorname{Th}(-E),H\Z] \cong H^n(B;\Z^\omega).
\end{align*}

Applying this to $E=\theta^*_n\gamma_{2n}$ gives a map $MT\theta_n \to \Sph$ which under Hurewicz maps to the Euler class of the oriented bundle $-\theta_n^*\gamma_{2n} \to BO(2n)\langle n \rangle$. Let $g\colon \Sigma^{\infty+1}\R P^{\infty}_{2n} \to \Sigma^{2n}\Sph$ be the composite 
\begin{align*}
    \Sigma^{\infty+1}\R P^\infty_{2n} \lto \Sigma^{\infty+1}SO/SO(2n) \lto \Sigma F_n \lto \Sigma^{2n+1}MT\theta_n \xrightarrow{e^S}\Sigma^{2n+1}\Sph.
\end{align*}
Define the spectrum $Y$ by the fiber sequence 
 \begin{align}\label{eq: defofY}
    Y \to \Sigma^{\infty+1} \R P^{\infty}_{2n} \stackrel{g}{\to} \Sigma^{2n+1}\Sph.
 \end{align}
\end{constr}

Now, we show that the map $\tau_{\leq 2n+k}MO\langle k+1 \rangle \to \tau_{\leq 2n+k}\Sigma^{\infty+1}\R P^\infty_{2n}$ for $0\leq k\leq n$ factors through $\tau_{\leq 2n+k} Y$.
\begin{lemma}\label{lemma:MOkliftY}
     For $0\leq k\leq n$, we may find a dashed lift in the diagram:
    \begin{equation*}
        \begin{tikzcd}
            & \tau_{\leq 2n+k}Y \dar{} \\
            \tau_{\leq 2n+k}MO\langle k+1 \rangle \rar{} \arrow[ur, dashed] & \tau_{\leq 2n+k}\Sigma^{\infty+1}\R P^\infty_{2n}.
        \end{tikzcd}
    \end{equation*}
\end{lemma}
\begin{proof}
    First, we show that the composition
    \[MO\langle k +1 \rangle \to \Sigma F_n \to \Sigma^{2n+1} MT\theta_n \xrightarrow{e^S} \Sigma^{2n+1} \Sph\]
    is null.
    For this, it suffices to note that the Euler class factors (the bundle on $\tau_{> n} BO(2n)$ is pulled back from $\tau_{>k} BO(2n)$):
    \begin{equation*}
        \begin{tikzcd}
            \Sigma^{2n+1}MT\theta_n \rar{e^S} \dar{} & \Sigma^{2n+1}\Sph \\\Sigma^{2n+1}MT\theta_{2n,k}  \arrow[ur, "e^S_k"'] &
        \end{tikzcd}
    \end{equation*}
    so that we are composing two arrows in a fiber sequence.

    Let $F_{2n,k}$ denote the fiber of $\Sigma^{2n} MT\theta_{2n,k} \to MO \langle k+1 \rangle$, and let $F_n = F_{2n,n}$.
    It follows that the map $MO\langle k + 1 \rangle \to \Sigma F_{2n,k}$ lifts to the fiber of the map $\Sigma F_n \to \Sigma^{2n+1} MT\theta_n \xrightarrow{e^S} \Sigma^{2n+1} \Sph$, which we denote by $X$.
    To conclude, we note that there is a diagram
    \begin{center}
        \begin{tikzcd}
            Y \ar[r] \ar[d] & X \ar[d] \\
            \Sigma^{\infty+1} \R P^\infty_{2n} \ar[r] & \Sigma F_{2n,k}
        \end{tikzcd}
    \end{center}
    and the horizontal arrows become equivalences upon applying $\tau_{\leq 2n+k}$ by Lemma \ref{lem:MTcofibseq}.
\end{proof}

Finally, we record the following result about the behavior of the map $g$ on bottom cohomology groups, which is a consequence of \cite[Lemma 5.3]{galatius2015abelian}. As a consequence, we prove that $\tau_{\leq \ell} Y$ is $2$-complete.

\begin{lemma} \label{lem:YeulerClass2}
    The map $\Z \cong H\Z^{2n+1}(\Sigma^{2n+1} \Sph) \xrightarrow{g^*} H\Z^{2n+1}(\Sigma^{\infty+1} \R P^{\infty}_{2n}) \cong \Z$ is given by multiplication by $\pm 2$.
    As a consequence, the bottom homotopy group of $Y$ is $\pi_{2n} Y \cong \Z/2$.
\end{lemma}
\begin{proof}
    To start, we note that $g^*$ factors as 
    \[H^{2n}(\Sigma^{2n} \Sph; \Z) \xrightarrow{(e^S)^*} H^{0}(MT\theta_n;\Z) \to H^{2n}(\R P^{\infty}_{2n}; \Z)\]
    by definition, and the map $(e^S)^*$ sends a generator to the Euler class in $H^{2n}(\tau_{>n}BSO(2n); \Z) \cong H^0(MT\theta_n;\Z),$
    where we use the Thom isomorphism to identify these cohomology groups.
    By \cite[Lemma 5.3]{galatius2015abelian}, the map $H^{2n}(BSO(2n);\Z) \to H^{2n}(SO/SO(2n);\Z)\cong H^{2n}(\R P^\infty_{2n};\Z)$ pulls back the Euler class to twice a generator.
    Combining this with the above, we deduce that $g^*$ sends a generator to twice a generator as desired.

    For the statement about $\pi_{2n} Y$, we combine the long exact sequence associated to $Y \to \Sigma^{\infty + 1} \R P^\infty_{2n} \xrightarrow{g} \Sigma^{2n+1} \Sph$
    with the Hurewicz theorem.
\end{proof}

\begin{corollary} \label{cor:Y2complete}
    The map $\tau_{\leq \ell} Y \to (\tau_{\leq \ell} Y)^{\wedge}_2$ is an equivalence for all $\ell$.
\end{corollary}
\begin{proof}
    Since $Y$ is of finite type, Postnikov truncation commutes with $2$-completion, so it suffices to show that $Y \to Y^{\wedge} _2$ is an equivalence.
    By the fracture square
    \begin{center}
    \begin{tikzcd}
        X \arrow[dr, phantom, "\usebox\pullback" , very near start, color=black] \rar{} \dar{}& X^{\wedge}_2 \dar{} \\
        X[\frac{1}{2}] \rar{} & X^{\wedge}_2 [\frac{1}{2}],
    \end{tikzcd}
    \end{center}
    it suffices to show that the map
    \[Y[\frac{1}{2}] \to Y^{\wedge} _2 [\frac{1}{2}] \]
    is an equivalence.

    In fact, both of these spectra are contractible. To see this, we note that the homology of $\Sigma^{\infty+1} \R P^\infty _{2n}$ and $(\Sigma^{\infty + 1} \R P^\infty_{2n})^{\wedge} _{2}$ with $\Z[\frac{1}{2}]$ coefficients is concentrated in degree $2n+1$, where it is given by $\Z[\frac{1}{2}]$ and $\mathbb{Q}_2$, respectively.
    It therefore follows from Lemma \ref{lem:YeulerClass2} that $Y$ and $Y^{\wedge} _2$ have trivial homology with $\Z[\frac{1}{2}]$ coefficients.
    Since $Y$ and $Y^{\wedge}_2$ are bounded below, it follows that $Y[\frac{1}{2}]$ and $Y^{\wedge} _2 [\frac{1}{2}]$ are contractible, as desired.
\end{proof}

\section{Splitting $MT\theta_n$ at $p=2$} \label{sec:splitting2}

It remains to prove that the attaching map $\tau_{\leq \ell} MT\theta_n \to \tau_{\leq \ell} Y$ is nullhomotopic.
Since we have proven that $\tau_{\leq \ell} Y$ is $2$-complete in Corollary \ref{cor:Y2complete} above, we may employ the $2$-primary Adams spectral sequence to prove this.
Our proof has two steps: first, we prove that the $E_2$-page of the Adams spectral sequence for $Y$ admits a vanishing line.
Second, we prove that the attaching map in question has high Adams filtration.
When combined, these results prove that it must be nullhomotopic.
%

\begin{definition}
    Let $M$ be a module over the mod $2$ Steenrod algebra $\mathcal{A}$. If 
\begin{align*}
\Ext^{s,t}_{\mathcal{A}}(M,\F_2) = 0
\end{align*}
for $ms>(t-s)+c$, we say that $M$ has a \emph{vanishing line} with \emph{slope} $\frac{1}{m}$ and \emph{intercept} $-c$. If $M$ arises as $M = H\F_2^*(X)$ for $X$ a spectrum, we say that $X$ has a vanishing line of slope $\frac{1}{m}$ and intercept $-c$ if this true for $M$.
\end{definition}


The first vanishing line makes its appearance in the following lemma (see also Figure \ref{fig:AdamsRPinf2mod2}):

\begin{lemma}\label{lem:vanishinglinestuntedproj}
    The spectrum $\Sigma^\infty\R P^\infty_{2n}$ has a vanishing line with slope $\frac{1}{2}$ and intercept $2n-3$, except in total degree $t-s=2n$, which consists of a single $h_0$-tower generated in filtration $0$.
\end{lemma}

\begin{proof}
    
From the cofiber sequence $\R P^{2n-1} \to \R P^\infty \to \R P^{\infty}_{2n}$ and the well-known cohomology of real projective spaces, we deduce 
\begin{align*}
H\F_2^q(\Sigma^\infty\R P^\infty_{2n}) = \begin{cases} \F_2 & \text{if }
        q\geq 2n \\
        0 & \text{else.}
    \end{cases}
\end{align*} 
Furthermore, the Steenrod squares act by $\Sq^i(x_{2n+k}) =\binom{2n+k}{i}x_{2n+k+i}$ where $x_{2n+k}$ is a generator of $H\F_2^{2n+k}(\Sigma^\infty\R P^\infty_{2n})$. In particular, we have $\Sq^1(x_{2n})=0$ and $\Sq^1 (x_{2n+k+1}) = x_{2n+k+2}$ for $k \geq 0$.
We may therefore fit $H\mathbb{F}_2^* (\Sigma^\infty \mathbb{R}P^{\infty} _{2n})$ into an exact sequence
\[0 \to \mathbb{F}_2[2n] \to H\mathbb{F}_2^* (\Sigma^\infty \mathbb{R}P^{\infty} _{2n}) \to M \to 0,\]
where $M$ is a free $\mathcal{A}_0=\Lambda(\Sq^1)$-module.
Since $M$ is free over $\mathcal{A}_0$, it follows from \cite[Theorem 2.1]{adams1966periodicity} that 
\begin{align*}
    \Ext^{s,t}_\A(M,\F_2)=0 \text{ when } t-s-(2n+1)+\varepsilon(s) < 2s, \quad \varepsilon(s) = \begin{cases}
    0 & \text{if } s\equiv 0\pmod{4} \\
    1 & \text{if } s\equiv 1\pmod{4} \\
    2 & \text{if } s\equiv 2,3 \pmod{4}.
\end{cases}
\end{align*}
On the other hand, it follows from \cite[Theorem 1.1]{adams1966periodicity} that $\Ext^{s,t}_\A(\F_2[2n],\F_2)=\Ext^{s,t-2n}_\A(\F_2,\F_2)$ consists of an $h_0$-tower starting in filtration zero for $t-s=2n$, and otherwise satisfies: 
\begin{align*}
    \Ext^{s,t-2n}(\F_2,\F_2) = 0 \text{ when } 0<t-s-2n+\eta(s) < 2s, \quad \eta(s) = \begin{cases}
    1 & \text{if } s\equiv 0,1 \pmod{4} \\
    2 & \text{if } s\equiv 2\pmod{4} \\
    3 & \text{if } s\equiv 3 \pmod{4}.
\end{cases}
\end{align*}
Combining these with the long exact sequence in Ext associated to
\[0 \to \mathbb{F}_2[2n] \to H\mathbb{F}_2^* (\Sigma^\infty \mathbb{R}P^{\infty} _{2n}) \to M \to 0,\]
we obtain the desired statement. Note that the boundary map entering the $t-s=2n$ Ext of $\mathbb{F}_2[2n]$ must be zero because the source is $h_0$-torsion (because of the vanishing line) and the target is $h_0$-torsionfree.
\end{proof}

\begin{figure}[h]
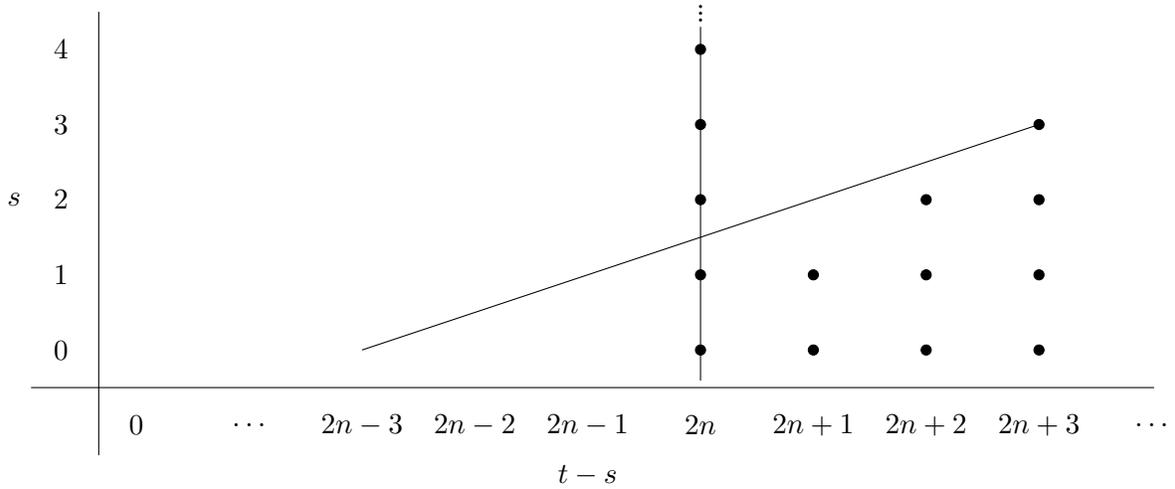

\begin{sseqpage}[ no x ticks, x range = {0}{8}, y range= {0}{4}, xscale = 1.5, classes = fill, x label = {$t-s$},
y label = {\rotatebox[origin=c]{270}{$s$}},  x axis extend end = 4em]
\begin{scope}[ background ]
\node at (\xmin,\ymin - 1) {0};
\node at (\xmin+1,\ymin - 1) {\dots};
\node at (\xmin+2,\ymin - 1) {2n-3};
\node at (\xmin+3,\ymin - 1) {2n-2};
\node at (\xmin+4,\ymin - 1) {2n-1};
\draw (\xmin+2,\ymin) -- (8,\ymin+3);
\node at (\xmax+1,\ymin - 1) {\dots};
\node at (\xmin +5 ,\ymin - 1) {\protect\vphantom{2}2n};
\pgfmathtruncatemacro{\newxmax}{\xmax - 5}
\foreach \n in {1,..., \newxmax}{
\node at (\n+5,\ymin - 1) {2n+\n};
}

\end{scope}
\foreach \n in {0,..., \ymax}{
\class(\xmin +5,\n)
}
\foreach \n in {4,..., \xmax}{
\class(\n+2,\ymin)
}
\foreach \n in {5,..., \xmax}{
\class(\n+1,\ymin+1)
}
\foreach \n in {6,..., \xmax}{
\class(\n+1,\ymin+2)
}

\foreach \n in {7,..., \xmax}{
\class(\n+1,\ymin+3)
}

\class(5,5)
\structline[shorten < = -3em](5,0)(5,5)

\end{sseqpage}
\caption{Illustration of the Adams $E_2$-page for $\Sigma^{\infty}\R P^\infty_{2n}$ with its vanishing line.}\label{fig:AdamsRPinf2mod2}
\centering
\end{figure}

We now show that if we replace $\Sigma^\infty \R P^\infty _{2n}$ by $Y$, we can improve the situation by eliminating the $h_0$-tower.

 \begin{lemma}\label{lem:vaninshinglineforY}
    The spectrum $Y$ has a vanishing line of slope $\frac{1}{2}$ and intercept $2n -3$.
 \end{lemma}
 \begin{proof}
 It follows from Lemma \ref{lem:YeulerClass2}
 that $g$ induces the zero map in mod $2$ cohomology and therefore the fiber sequence \eqref{eq: defofY} induces a short exact sequence of $\mathcal{A}$-modules:
 \begin{align}\label{eq:SESofAmods}
     0 \to H\F_2^*(\Sigma^{\infty+1} \R P^{\infty}_{2n}) \to H\F_2 ^*(Y) \to H\F_2^*(\Sigma^{2n}\Sph) \to 0.
 \end{align}
%

Now, by examining the long exact sequence on Ext associated to \eqref{eq:SESofAmods} and again applying \cite[Theorem 1.1]{adams1966periodicity}, we see that the only possible contributions to $\Ext^{s,t}_\A(H\F_2^{*}(Y),\F_2)$ which could lie above the desired vanishing line come from the $h_0$-towers in total degrees $t-s=2n$ and $t-s=2n+1$ of in the Ext of $H\F_2^*(\Sigma^{2n}\Sph)$ and $H\F_2^*(\Sigma^{\infty+1} \R P^{\infty}_{2n})$, respectively.
We claim that the boundary map sends a generator for the $h_0$ tower of $H\F_2^*(\Sigma^{\infty+1} \R P^{\infty}_{2n})$ to $h_0$ times a generator for the tower of $H\F_2^*(\Sigma^{2n}\Sph)$.
Now, Lemma \ref{lem:YeulerClass2} states that $\pi_{2n} Y \cong \Z/2$.
In particular, it follows from the Adams spectral sequence that $h_0$ times a generator for the tower of $H\F_2^*(\Sigma^{2n}\Sph)$ must be zero on the $E_\infty$-page for $Y$.
The only way this can happen is if it is killed by the boundary map as described above.
 \end{proof}

We are going to show that the attaching map $\tau_{\leq 2n +k} MO\langle k+1 \rangle \to \tau_{\leq 2n+k} Y^{\wedge}$ is zero, possibly after further truncation, by an Adams filtration argument.
We start by noting that, for bounded below finite type $p$-complete spectra, there is a particularly convenient notion of skeleton.
%
%

\begin{constr}
    Given a bounded-below finite-type $p$-complete spectrum $E$, we construct an $r$-skeleton $\sk_{\leq r} E \to E$ such that:
    \begin{itemize}
       \item $\sk_{\leq r} E$ is a finite $p$-complete spectrum.
       \item the map $E \to \sk_{\leq r} E$ induces an isomorphism on mod $p$ homology in degrees $\leq r$.
       \item $H_* (\sk_{\leq r} E; \F_p) = 0$ for $* > r$.
    \end{itemize}
    To construct $\sk_{\leq r} E \to E$, we induct on $r$.
    Given
    $\sk_{\leq r-1} E \to E$,
    we let $C$ denote the cofiber and choose a basis for $H_r (C; \F_p),$ which is finite-dimensional by assumption.
    By Hurewicz, we may lift this basis to a map
    \[\bigoplus (\Sph_p^\wedge) ^r \to C.\]
    We then define $\sk_{\leq r} E$ to be the cofiber of
    $\bigoplus (\Sph_p^\wedge) ^{r-1} \to \Sigma^{-1} C \to \sk_{\leq r-1} E$.

    We note that this construction is \emph{weakly functorial} in the sense that a map $f : E \to F$ of bounded below finite type $p$-complete spectra may always be fit into a (non-functorial!) diagram
    \begin{center}
    \begin{tikzcd}
      \sk_{\leq r} E \ar[r] \ar[d,"\sk_{\leq r} f"] & E \ar[d,"f"] \\
       \sk_{\leq r} F \ar[r] & F.
    \end{tikzcd}
    \end{center}
\end{constr}

\begin{conv}
    From now until the end of this section, we implicitly assume that all spectra are $2$-completed.
\end{conv}

\begin{lemma}\label{lem:vanishinglineskeleton}
    If $X$ is a connective spectrum with a vanishing line of slope $\frac{1}{m}$ and intercept $c$, then for any bounded below finite type spectrum $E$ the Adams $E_2$-page of $[\sk_{\leq r}E, X]_*$ has a vanishing line of slope $\frac{1}{m}$ and intercept $c-r$.
\end{lemma}
\begin{proof}
The mod $2$ cohomology of $sk_{\leq r} E$ lies in degrees $k \leq * \leq r$ for some $k > -\infty$ by definition.
Equipping $H\F_2^{*}(\sk_{\leq r} E)$ with the degree filtration and considering the resulting spectral sequence, we see that it suffices to establish a vanishing line of the desired form for the groups
\[\Ext^{s,t}_{\mathcal{A}}(H^*(X),\F_2[k])=\Ext^{s,t-k}_{\mathcal{A}}(H^*(X),\F_2),\]
when $k \leq r$, which follows from the assumption.
\end{proof}

Applying Lemma \ref{lem:vanishinglineskeleton} to $E=MO\langle k+1\rangle$ and $X=Y$, we obtain the following result:
%

\begin{lemma}\label{lem:filtrationnullhom}
    A map $\sk_{\leq r}MO\langle k+1 \rangle \to Y$ of Adams filtration $\geq d$ is nullhomotopic whenever $r \leq  2(d+k-2)$.
\end{lemma}
\begin{proof}
By Lemma \ref{lem:vaninshinglineforY} and Lemma \ref{lem:vanishinglineskeleton}, the $E_2$-page of the Adams spectral sequence for $[\sk_{\leq r} MO \langle k+1 \rangle, Y]_*$ has a vanishing line of slope $\frac{1}{2}$ and intercept $2k-3-r$.
It therefore follows that in total degree $0$, the $E_2$-page is zero in filtration $2d > r+3-2k$, which is equivalent to the condition $r \leq 2(d+k-2)$.
The lemma follows from strong convergence of the Adams spectral sequence.
%
%
%
%
\end{proof}

Finally, we will also find the following lemma useful:

\begin{lemma}\label{lem:nullonskeleton}
    Let $E$ be a bounded below finite type spectrum. Then the composition  $\sk_{\leq \ell}E \to E \to \tau_{\leq \ell}E$ induces an injection
    \begin{align*}
        [\tau_{\leq \ell}E, \tau_{\leq \ell}F] \hookrightarrow [\sk_{\leq \ell}E, \tau_{\leq \ell}F].
    \end{align*}
\end{lemma}
\begin{proof}
    The map $\sk_{\leq \ell} E \to E \to \tau_{\leq \ell} E$ is an $\ell$-equivalence, so the fiber is $\ell$-connective. Injectivity now follows from the long exact sequence after applying $[-,\tau_{\leq \ell} F]$ and the usual $t$-structure on spectra implies $[\Sigma \tau_{\geq \ell}X,\tau_{\leq \ell}Y]=0$ for any spectra $X$ and $Y$.
\end{proof}

Now we put everything together, and figure out the exact range in which the attaching map $\tau_{\leq \ell} MO \langle n+1 \rangle \to \tau_{\leq \ell} Y$ is nullhomotopic.

\begin{proposition}\label{prop:ellnullp2} 
    The map $\tau_{\leq\ell}MO\langle n+1\rangle \to \tau_{\leq\ell}Y$ is nullhomotopic for $\ell \leq 2n+\floor{\frac{n}{2}}-5$.
\end{proposition}

\begin{proof}
%
Putting together the cofiber sequences from Lemma \ref{lem:MTcofibseq}, we contemplate the commutative diagram  
\begin{equation}\label{eq:nullhtpyonskeleton}
\begin{tikzcd}
	{\sk_{\leq \ell}MO\langle n+1 \rangle} & {MO\langle n+1 \rangle} & {\tau_{\leq 3n-k}MO\langle n+1 \rangle} & {\tau_{\leq 3n-k}Y} \\
	\vdots & \vdots & \vdots & \vdots \\
	{\sk_{\leq \ell}MO\langle n-k+1 \rangle} & {MO\langle n-k+1 \rangle} & {\tau_{\leq 3n-k}MO\langle n-k+1 \rangle} & {\tau_{\leq 3n-k}Y}
	\arrow[from=1-1, to=1-2]
	\arrow[from=1-2, to=1-3]
	\arrow[from=1-3, to=1-4]
	\arrow[from=1-1, to=2-1]
	\arrow[from=2-1, to=3-1]
	\arrow[from=3-1, to=3-2]
	\arrow[from=2-1, to=2-2]
	\arrow[from=2-2, to=2-3]
	\arrow[from=1-3, to=2-3]
	\arrow[from=1-2, to=2-2]
	\arrow[from=3-2, to=3-3]
	\arrow[from=3-3, to=3-4]
	\arrow[equal, from=1-4, to=2-4]
	\arrow[from=2-3, to=2-4]
	\arrow[from=2-2, to=3-2]
	\arrow[from=2-3, to=3-3]
	\arrow[equal, from=2-4, to=3-4]
\end{tikzcd}
\end{equation}
By Lemma \ref{lem:nullonskeleton}, the proposition follows by establishing the upper horizontal composite is nullhomotopic for $\ell \leq 2n+\lfloor\frac{n}{2} \rfloor - 5$.

In the first place, if 
\begin{equation}
    \ell \leq 3n-k \tag{\dag}\label{eq:dag}
\end{equation}
then the map $\sk_{\leq \ell} MO\langle n-k+1 \rangle \to \tau_{\leq 3n-k} Y$ lifts uniquely up to homotopy to a map $\sk_{\leq \ell} MO\langle n-k+1 \rangle \to Y$.
We therefore have a diagram of the form
\begin{center}
    \begin{tikzcd}
        \sk_{\leq \ell} MO\langle n+1 \rangle \ar[d] \ar[r] & Y \\
        \vdots \ar[d] & \\
        \sk_{\leq \ell} MO\langle n -k+1 \rangle \ar[uur]
    \end{tikzcd}
\end{center}
and we wish to show that the map $\sk_{\leq \ell} MO \langle n \rangle \to Y$ is nullhomotopic.

By the Thom isomorphism and \cite{stongBOk} the maps $MO\langle s+1\rangle \to MO\langle s \rangle$ for $s\equiv 0,1,2,4 \mod 8$ are trivial in $H\F_2^q$ for $q<2^{h(s)}$ where $h(s)\defeq \{m\in \N \mid 0<m\leq s, m \equiv 0,1,2,4 \mod 8\}$. Hence if we choose $\ell$ such that
\begin{equation}
\ell <2^{h(n-k+1)}, \tag{$\ast$}\label{eq:ast}
\end{equation}
then the maps $\sk_{\leq \ell}MO\langle s+1 \rangle \to  \sk_{\leq \ell }MO\langle s \rangle$ will be trivial in $H\F_2^*$-cohomology whenever $s \geq n-k+1$ and $s \equiv 0,1,2,4 \mod 8.$

Thus, if we choose $\ell$ to satisfy \eqref{eq:ast} and \eqref{eq:dag}, we obtain a map $\sk_{\leq \ell}MO\langle n \rangle \to Y$ with Adams filtration $\geq h(n,k) \defeq h(n) - h(n-k)$. Applying Lemma \ref{lem:filtrationnullhom} we see that this map is nullhomotopic whenever
\begin{equation}
    \ell \leq 2h(n,k)+2n-4. \tag{\ddag}\label{eq:ddag}
\end{equation}
Combining the three conditions \eqref{eq:ast}, \eqref{eq:dag} and \eqref{eq:ddag}, we need to maximize the integer $\ell$ satisfying 
\begin{equation*}
    \ell \leq \min\left(2^{h(n-k+1)}-1,  3n-k , 2h(n,k)+2n-4\right)
\end{equation*}
as a function of the integer $k$. 
In fact, we will first maximize $\ell$ subject to \eqref{eq:dag} and \eqref{eq:ddag}, and show that it automatically satisfies \eqref{eq:ast}, except for a finite number of cases that we do by hand.

To this end, we note the following elementary bounds $\floor{\frac{n}{2}}+2\geq h(n) \geq \floor{\frac{n}{2}}$ and $x\geq m\floor{\frac{x}{m}}\geq x-m$. Then
\begin{align*}
    2n-4 +2h(n,k) \geq 2n-4+2\left(\floor{\frac{n}{2}}-\floor{\frac{n-k}{2}}-2\right) \geq 2n+k-10
\end{align*}
Since the \eqref{eq:dag} function is decreasing in $k$ and the \eqref{eq:ddag} function is increasing in $k$ their minimum is maximized when equality is achieved, which yields 
\begin{align*}
    k= \frac{n}{2}+5.
\end{align*}
Since $k$ has to be an integer, we know that either the floor or ceiling of the above maximizes $\ell$. To figure out which it is, we plug the floor into the increasing function and the ceiling of $k$ into the decreasing function. It turns out that this gives the same value, namely 
\begin{align*}
    \ell = 2n+\floor{\frac{n}{2}}-5,
\end{align*}
It remains to see that this choice of $k$ also satisfies \eqref{eq:ast}, except for finitely many small values of $n$. This is done by finding upper bounds (resp. lower bounds) for the left (resp. right) hand side of \eqref{eq:ast} until we get rid of all floor and ceiling functions:
\begin{align*}
\ell \leq 2n + \frac{n}{2}-5 \stackrel{?}{\leq}  2^{\frac{3n}{5}-4} -1 \leq 2^{h(n-k+1)}-1
\end{align*}
The middle inequality only holds for $n\geq 16$, and so we determine $\ell$ for $n\leq 15$ by hand. The following table verifies that our choice of $\ell$ also works for these small values:
\begin{table*}[h]
    \centering
    \[
    \begin{array}{c| *{23}{c}}
n &  \, \, \, \, \, 0 & 1 & 2 & 3 & 4 & 5 & 6 & 7 & 8 & 9 & 10 & 11 & 12 & 13 & 14 & 15  \\
\hline 
\ell &  -1 & 0 & 1 & 3 & 3 & 7 & 7 & 7 & 7 & 15 & 19 & 21 & 25 & 27 & 29 & 31
    \end{array}\]
    \caption{Determining the maximum $\ell$ satisfying \eqref{eq:ast}, \eqref{eq:dag} and \eqref{eq:ddag} for small $n$.}
    \label{tab:smallnsplitting}
\end{table*}

\end{proof}

Putting everything together, we prove Theorem \ref{thm:mainthm}.

\begin{proof}[Proof of Theorem \ref{thm:mainthm}]
    Combine Corollary \ref{cor:getSplit}, Corollary \ref{cor:Y2complete} and Proposition \ref{prop:ellnullp2}.
\end{proof}

\section{Applications to $B\Diff(W^{2n}_g, D^{2n})$}
In this section we outline an application of Theorem \ref{thm:mainthm}: the calculation of $H_2(B\Diff(W^{2n}_g, D^{2n});\Z)$ for large $n$ and $g$.

\subsection{Low Degree Homology of $B\Diff(W^{2n}_g, D^{2n})$}  \label{sec:H2}
To translate from homotopy to homology we will use the fact that the first $k$-invariant of the infinite loop space $X=\Omega^\infty_0 MT\theta_n$ vanishes by the $n=2$ case of \cite[Corollary 2.5]{arlettaz_1988}. Then the second stage $\tau_{\leq 2} X$ in the Postnikov tower for $\Omega^\infty_0 MT\theta_n$ decomposes as
\begin{align*}
    \tau_{\leq 2} X \simeq K(\pi_1(X),1) \times  K(\pi_2(X),2).
\end{align*}
Applying relative Hurewicz to the $3$-connected map $X \to \tau_{\leq 2} X$ gives $H_2(X)\cong H_2(\tau_{\leq 2} X)$. The Künneth theorem then implies 
\begin{align*}
    H_2(X) \cong H_2(\pi_1(X)) \oplus \pi_2(X),
\end{align*}
using that homology of a $K(G,1)$ is group homology of $G$. The component $\pi_1(X)$ was computed in \cite[Theorem 1.3]{galatius2015abelian}: The authors construct a map
\begin{align*}
    \pi_1(\Omega^\infty_0 MT\theta_n) \to \pi_{2n+1}(MO\langle n+1 \rangle) \oplus H_1(\operatorname{Aut}(Q_{W_{g}^{2n}}))
\end{align*}
which is an isomorphism for $n\neq 2$ and $g \geq 5$. Here $Q_{W_{g}^{2n}}$ is a quadratic form whose underlying bilinear form is the intersection form on $H_n(W_{g}^{2n})$, and \cite{galatius2015abelian} proves that $H_1(\operatorname{Aut}(Q_{W_{g}^{2n}})) \cong \pi_{2n+1}(\Sigma^\infty SO/SO(2n)) \cong \pi_{2n+1} (\Sigma^\infty \R P^\infty _{2n})$ under these assumptions and that 
\[H_1(\operatorname{Aut}(Q_{W_{g}^{2n}})) \cong \begin{cases}
    (\Z/2)^2 & \text{if } n \text{ even,} \\
    0 & \text{if } n = 1,3,7, \\
    \Z/4 & \text{otherwise.}
\end{cases}\]
We note that this isomorphism can also be deduced from Theorem \ref{thm:mainthm}, albeit with a worse bound on $n$.

Thus we are reduced to understanding $\pi_2(\Omega^{\infty}_0MT\theta_n)\cong\pi_2(MT\theta_n)$. This group is determined by Theorem \ref{thm:mainthm} for $n\geq 16$, where we have a splitting $\pi_2(MT\theta_n) \cong \pi_{2n+2} MO \langle n+1 \rangle \oplus \pi_{2n+2} \Sigma^{\infty} \R P^\infty _{2n}$. The map 
$\alpha\colon B\D(W^{2n}_g,D^{2n}) \to \Omega^\infty_0 MT\theta_n$ is an isomorphism on $H_2$ whenever $g \geq 7$. Combining these two observations yields the following calculation.

\begin{corollary}\label{cor:H2MTtheta}
For $n\geq 16$ and $g\geq 7$ we have a non-canonical isomorphism    
\begin{align*}
   H_2(B\D(W^{2n}_g,D^{2n});\Z)  &\cong H_2(\pi_{2n+1} MO \langle n+1 \rangle\oplus \pi_{2n+1}(\Sigma^{\infty} \R P^\infty _{2n})) \\
   &\oplus \pi_{2n+2} MO \langle n+1 \rangle\oplus \pi_{2n+2}(\Sigma^\infty \R P^\infty _{2n}),
\end{align*}
\end{corollary}
We can in many cases of interest determine the cobordism groups of highly connected manifolds appearing above. Starting with the exact sequence of spectra
\[
\Sph \lto MO\langle n+1 \rangle \lto MO\langle n+1 \rangle / \Sph 
\]
Stolz proves in \cite[\S 3]{stolz2007hochzusammenhangende} that $\pi_* MO\langle n+1 \rangle /\Sph \simeq \pi_*bo\langle n+1 \rangle \oplus \pi_*A[n+1]$ where $A[n+1]$ is a $2(n+1)$-connective spectrum satisfying $A[n+1] \simeq A[n+2]$ for $n+1 \equiv 3,5,6,7 \mod 8$ \cite[Satz 3.1 (ii), (iii)]{stolz2007hochzusammenhangende}. The splitting on homotopy groups is a consequence of \cite[Lemma 3.7]{stolz2007hochzusammenhangende}. Moreover, on the $bo \langle n+1 \rangle$ term, this identifies the boundary map $\partial : \pi_{*} MO\langle n+1 \rangle / \Sph  \to \pi_{*-1} \Sph$ with the $J$-homomorphism. Then the series of papers \cite{burklund2019boundaries,burklund2020highdimensional, sengeryiyu2022} establish that $\partial(\pi_{2n+2} A[n+1])\subset \im(J)_{2n+1}$ when $n+1 \neq 9,12$ and $\partial(\pi_{2n+3} A[n+1])\subset \im(J)_{2n+2}$ when $n+1 \neq 1,3,4,7,8,9$. (Though many cases of this were proven earlier, for example in the work of Stolz \cite{stolz2007hochzusammenhangende}.)

So outside of these cases we have the following short exact sequences
\[\begin{tikzcd}
	0 & {(\coker(J))_i} & {\pi_i(MO\langle n+1 \rangle)} & {\ker(\partial)_i} & 0,
	\arrow[from=1-1, to=1-2]
	\arrow[from=1-2, to=1-3]
	\arrow[from=1-3, to=1-4]
	\arrow[from=1-4, to=1-5]
\end{tikzcd}\]
\[\begin{tikzcd}
	0 & {(\ker(J))_i} & {\ker(\partial)_i} & {\pi_i(A[n+1 ])} & 0,
	\arrow[from=1-1, to=1-2]
	\arrow[from=1-2, to=1-3]
	\arrow[from=1-3, to=1-4]
	\arrow[from=1-4, to=1-5]
\end{tikzcd}\]
for $i=2n+1$ and $i=2n+2$. It follows from Serre's finiteness theorem and Adams's work on the $J$-homomorphism that $\ker(J)_i$ is $0$ if $i \not\equiv 0 \pmod{4} $ and $\Z$ otherwise.  

In the case where $i=2n+1$, we know that $\pi_{2n+1}(A[n+1]) \cong 0$ as it is $2(n+1)$-connective which implies that $\ker(J)_{2n+1} \cong \ker(\partial)_{2n+1} \cong 0$. Hence we recover the fact that $\pi_{2n+1}(MO\langle n+1\rangle) \cong \coker(J)_{2n+1}$ outside of the exceptional $n+1 =9,12$.

In the $i=2n+2$ case, the situation is  not as straightforward. We know by the above properties of $A[n+1]$ and \cite[Theorem A]{stolz2007hochzusammenhangende} that
\[
\pi_{2n+2}(A[n+1]) \cong \begin{cases}
    \Z & \text{if } n+1 \equiv 0,4 \pmod{8} \\
    \Z/2 & \text{if } n+1 \equiv 1,2 \pmod{8} \\
    0 & \text{if } n+1 \equiv 3,5,6,7 \pmod{8}. 
\end{cases}
\]
Hence we similarly conclude that 
\[
\pi_{2n+2}MO\langle n+1\rangle \cong
\begin{cases}
    \coker(J)_{2n+2}\oplus \Z & \text{if } n+1 \equiv 6 \pmod{8} \\
    \coker(J)_{2n+2} & \text{if } n+1 \equiv 3,5,7 \pmod{8} \\
    \coker(J)_{2n+2}\oplus \Z^2 & \text{if } n+1 \equiv 0,4 \pmod{8} \\
    E_1 & \text{if } n+1 \equiv 1 \pmod{8} \\
    E_2 & \text{if } n+1 \equiv 2 \pmod{8}.
\end{cases}
\]
Here the group $E_1$ is determined up to the following extension
\[\begin{tikzcd}
	0 & {(\coker(J))_i} & {E_1} & {\Z/2}& 0,
	\arrow[from=1-1, to=1-2]
	\arrow[from=1-2, to=1-3]
	\arrow[from=1-3, to=1-4]
	\arrow[from=1-4, to=1-5]
\end{tikzcd}\]
whereas $E_2$ is determined up to the two extensions
\[\begin{tikzcd}
	0 & {\Z} & {E_3} & {\Z/2} & 0,\\
 0 & {\coker(J)_{2n+2}} & {E_2} & {E_3} & 0.
	\arrow[from=1-1, to=1-2]
	\arrow[from=1-2, to=1-3]
	\arrow[from=1-3, to=1-4]
	\arrow[from=1-4, to=1-5]
	\arrow[from=2-1, to=2-2]
	\arrow[from=2-2, to=2-3]
	\arrow[from=2-3, to=2-4]
	\arrow[from=2-4, to=2-5]
\end{tikzcd}\]
The homotopy groups $\pi_{2n+2}(\Sigma^\infty \R P^\infty _{2n})$ are calculated in \cite[Table 1, $k=2n$, $p=2$]{hoo1965some}
\[
\pi_{2n+2}(\Sigma^\infty \R P^\infty_{2n})= \begin{cases} (\Z/2)^2&\text{if } 2n\equiv 0,4 \pmod{8} \\
0&\text{if } 2n\equiv 2,6 \pmod{8} \\    
\end{cases}
\]

Combining this with the above analysis and the fact that $H_2$ of any finite cyclic groups vanishes and Künneth, we provide a table with some example values of $H_2(B\Diff(W_{g}^{2n}, D^{2n});\Z)$ in terms of the above splitting.


\begin{table*}[h]
    \centering
    \[
    \begin{array}{c| c c c }
n & \quad \quad \quad \quad  16 \quad \quad \quad \, \, \, \,  & \, \,\,\quad \quad \quad \quad 17\quad \quad &  18 \\
\hline 
H_2 & E_1\oplus (\Z/2)^{15} & (\Z/2)^7\oplus \Z/2^2\oplus \Z/3\oplus \Z/5 & \, \Z/2\oplus \Z/2^2\oplus E_2 
    \end{array}\]
    \caption{Table of $H_2(B\Diff_\partial(W_{g}^{2n}, D^{2n});\Z)$}
    \label{tab:H2BDiff2}
\end{table*}

\bibliographystyle{alpha}
\bibliography{example}

\end{document}